\documentclass{amsart}
\usepackage{graphicx}
\usepackage{amsmath}
\usepackage{amsfonts}
\usepackage{amssymb}
\newtheorem{theorem}{Theorem}[section]
\newtheorem{lemma}[theorem]{Lemma}

\theoremstyle{definition}
\newtheorem{definition}[theorem]{Definition}

\theoremstyle{remark}
\newtheorem{remark}[theorem]{Remark}
\newtheorem{corollary}[theorem]{Corollary}

\newtheorem{proposition}[theorem]{Proposition}
\numberwithin{equation}{section}

\begin{document}
	
	\title[Bivariate Countable FIF and Nonlinear Fractal Operator]{On Bivariate Fractal Interpolation  for Countable Data and Associated Nonlinear Fractal Operator}
	\author{K. K. Pandey}
	\address{Department of Mathematics, IIT Delhi, New Delhi, India 110016}
	
	\email{kshitij.sxcr@gmail.com}
	\author{P. Viswanathan}
	\address{Department of Mathematics, IIT Delhi, New Delhi, India 110016}
	
	\email{viswa@maths.iitd.ac.in}
	
	
	
	
	
	
	
	\begin{abstract}
We provide a general framework to construct fractal interpolation surfaces (FISs)  for a prescribed  countably infinite data set  on a rectangular grid. Using this as a crucial tool, we obtain  a parameterized family of bivariate fractal functions simultaneously interpolating and approximating  a prescribed bivariate continuous function.  Some elementary properties of the associated nonlinear (not necessarily linear) fractal operator are established, thereby initiating the interaction of the notion of fractal interpolation with the theory of nonlinear operators.  
	\end{abstract}
	
	\maketitle

	\section{INTRODUCTION}
	Over the past three decades, the subject of fractal  interpolation has  been one of the major research themes among fractals community.  The concept of Fractal Interpolation Function (FIF) was introduced by Barnsley in reference \cite{MF1},  where  he demonstrates the existence of a univariate continuous function  interpolating a given data set such that the graph of the constructed  interpolant is  a fractal in the sense that it is the attractor of a suitable Iterated Function System (IFS) \cite{JH}.
We do not attempt here to provide a complete list of references on fractal interpolation, as the field is quite large. Instead, we  refer the reader to the recent research works on fractal interpolation \cite{DO,Luor,PM1,SR,WY} and references therein. For a compendium  of fractal interpolation and related topics, the reader may  consult the interesting monograph \cite{PM2}. 

\par As is well known, interpolation and approximation are intimately related, however the interplay between these two theories is more subtle in the fractal setting.  It is our opinion that in the context of  fractal interpolation,   $\alpha$-fractal function - a notion explored by Navascu\'{e}s  \cite{MN1,MN2} - is mostly responsible for the favourable findings on the interconnection between interpolation and approximation theories of univariate functions.  In fact, $\alpha$-fractal function provides a parameterized family of fractal functions  that simultaneously interpolate and approximate a given univariate continuous function. The parameters can be adjusted so that the fractal functions share or modify the properties of the original function, for instancce, smoothness and shape properties \cite{SC, VN2}.  The analytical properties of the fractal operator that maps every function to its fractal counterpart  and new approximation classes of self-referential functions obtained as fractalization of various approximation classes (such as  polynomials, rational functions, and trigonometric polynomials) via this fractal operator have received considerable attention in the literature on univariate fractal approximation theory \cite{MN1,MN2,MN3,VC, VN}. Let us stress here that the studies on the aforementioned fractal operator are mostly  confined to the realm of bounded linear operators. 
	\par
	As with any such salient idea, numerous questions and results based on fractal interpolation were spawned. One question asked had to do with its extension to multivariate case.   In this direction, several works have been done to interpolate a given bivariate data using bivariate FIFs (fractal surfaces), see, for instance, \cite{GP,Dalla,Feng,GH,Mal,PM,Metz,Xie,Zhao} and a few researches deal with multivariate FIFs \cite{BD,HM}. Among various constructions  available in the literature,  we found the general framework to construct fractal surfaces given in \cite{Ruan} to be  quite interesting. This is due to the fact that the construction thereat is amenable to obtain bivariate analogue of $\alpha$-fractal function, which is a natural entry point to delve into the theory of  bivariate fractal approximation, see also \cite{SCNA, SV}. 
	\par
	Much of the existing literature dealing with univariate and multivariate FIFs concentrate  primarily on finite data sets. This is no happenstance, as, the theory of fractal interpolation is based on Hutchinson's fundamental  result that guarantees the existence of invariant set (attractor) for a set-valued map induced by a finite number of contractions \cite{JH}.   Attempts were made in literature to define these concepts in the setting of countably infinite maps. For instance, Secelean adapted  Hutchinson approach so as to handle countable number of contractions \cite{NAS1} and    deduced the existence of a univariate fractal function interpolating a data set consisting of countably infinite points \cite {NAS2}.  Bivariate fractal interpolation function for an infinite sequence of data is not studied hitherto and naturally we want to generalize  construction of fractal surfaces so as to accomodate  countably infinite data set.  Let us note that  this is practically important bacause, for instance, in the theory of sampling and reconstruction, often one works with infinite sequence of data points and a general approach focuses to seek an approximate rather than the perfect reconstruction on some restricted class of signals. 
	\par
   One of the objectives of this note is to bring to light a construction of Fractal Interpolation Surface (FIS) for an arbitrary  countably infinite data set over rectangular grids.  To this end, we borrow the construction from \cite{Ruan},  and apply it, \textit{mutatis mutandis}, to the setting of countable data.       This part of the current findings may be viewed also as a sequel to \cite{NAS2}, where the univariate fractal interpolation for a countable data set was studied. As hinted earlier, among various constructions of FISs, our choice of \cite{Ruan} is guided by the fact that    it offers a simple and efficient  platform to obtain  a parameterized family of fractal functions corresponding to a prescribed bivariate continuous function.   Crucial to further development will be  this parameterized  family consisting of approximate fractal reconstructions of the original  function (seed function) defined on a rectangular domain. This family of fractal functions  is obtained by sampling the seed function  at infinite number of grid points in the domain and applying the countable FIS scheme developed in the first part of this note. 
\par We define a fractal operator that sends each seed function to its fractal approximate reconstruction and study some analytical properties of this operator, which is, in general, nonlinear. Though the bounded linear fractal operator has its origin in the theory of univariate $\alpha$-fractal function (see, for example, \cite{MN2}),  our focus will be to its intriguing links with the perturbation theory of operators (not necessarily linear or bounded). Thus,  the research reported  here could open the door for intense and fruitful interaction of  two fields - fractal interpolation and the theory of nonlinear operators,  and the potential  applications lie, for instance,  in the field of sampling and reconstruction. 
    
	\section{Review of countable iterated function system and bivariate fractal interpolation}
	As hinted in the introductory section, impetus for carrying out the studies reported in this note came from the reading of \cite{Ruan, NAS2}. For the sake of convenience, we outline a few important concepts from these references that form the background material. 
	\begin{definition}
		Let $(X,d)$ be a compact metric space and $(\omega_i)_{i\in\mathbb{N}}$ be a sequence of continuous self maps on $X.$ The set $\{X,\omega_i: i \in \mathbb{N}\}$ is called  a countable iterated function system (CIFS). The CIFS is said to be hyperbolic if the functions, $\omega_i$, $i\in \mathbb{N}$ are contractions, say,  with respective contractivity factor $r_i$ such that $\sup_{i \in \mathbb{N}}r_i < 1.$ 
	\end{definition}
\begin{definition}
	A non-empty set $A\subset X$ is said to be a set fixed point of the CIFS if
	$A = \overline{\bigcup_{i\in \mathbb{N}}\omega_i(A)}.$
	If the set fixed point is unique, then it is said to be the attractor of the CIFS.
\end{definition}
\begin{theorem}[\cite{NAS1}] \label{Basthm1}
Let $\{X,\omega_i: i \in \mathbb{N}\}$ be a hyperbolic CIFS. Then there exists a non-empty compact set $A\subset X$ such that
$A = \overline{\bigcup_{i\in \mathbb{N}}\omega_i(A)}.$
\end{theorem}

Let $\{(x_i,y_j,z_{ij}) \in \mathbb{R}^3 : i=0,1,\dots,m; j=0,1,\dots,n\}$ be an interpolation data set such that $a=x_0< x_1 < \dots <x_m=b$ and $c=y_0<y_1<\dots <y_n=d.$ Set $I=[a,b]$ and $J=[c,d].$  For brevity, we shall make the following abbreviations. Let
$\Sigma_m=\{1,2,\dots,m\}$, $\Sigma_{m,0}=\{0,1,\dots m\}$,  $\partial \Sigma_{m,0}=\{0,m\}$ and 
	int$\Sigma_{m,0}=\{1,2,\dots,m-1\}.$ Similarly, we can define $\Sigma_n$, $\Sigma_{n,0}$, $ \partial \Sigma_{n,0}$ and int$\Sigma_{n,0}.$ Denote $I_i=[x_{i-1},x_i]$ and $J_j=[y_{j-1},y_j]$ for $i \in \Sigma_m$ and $j \in \Sigma_n.$ For any $i \in \Sigma_m,$ let $u_i:I \rightarrow I_i$ be a contractive homeomorphism satisfying 
\begin{equation*}
\begin{cases}
& u_i(x_0)=x_{i-1}, u_i(x_m)=x_i, ~~\text{if i is odd},\\
& u_i(x_0)=x_i, u_i(x_m)=x_{i-1},~~ \text{if i is even},\\
& |u_i(x)-u_i(x')| \le a_i|x -x'|, ~~~~ \forall ~~x, x' \in I,
\end{cases}
\end{equation*}
where $0 < a_i < 1$ is a  constant. Similarly, for any $j \in \Sigma_n,$ let $v_j:J \rightarrow J_j$ be a contractive homeomorphism satisfying
\begin{equation*}
\begin{cases}
v_j(y_0)=y_{j-1}, v_j(y_n)=y_j,\text{ if j is odd},\\
v_j(y_0)=y_j, v_j(y_n)=y_{j-1}, ~~ \text{if j is even},\\
|v_j(y)-v_j(y')| \le b_j|y -y'|, ~~~~ \forall ~~y, y' \in J, 
\end{cases} 
\end{equation*}
where $0 < b_j < 1$ is a constant. By the definitions of $u_i$ and $v_j,$ it is easy to check that
\begin{equation*}
\begin{aligned}
u^{-1}_i(x_i)=u^{-1}_{i+1}(x_i), ~~~~\forall ~~i \in \text{int}\Sigma_{m,0},
\end{aligned} 
\end{equation*}
and
\begin{equation*}
\begin{aligned}
v^{-1}_j(y_j)=v^{-1}_{j+1}(y_j),~~~~ \forall ~~j \in \text{int}\Sigma_{n,0}. 
\end{aligned} 
\end{equation*}
Let $\tau: \mathbb{N}\times \{0,m,n\} \rightarrow \mathbb{N}\cup\{0\}$ be defined by 
\begin{equation*}
\tau(i,0)=
\begin{cases}
 i-1,  \text{ if}~~  i ~~\text{is odd},\\ 
 i, \text{ if} ~~ i~~ \text{is even}.
\end{cases} 
\end{equation*}
\begin{equation*}
\tau(i,m)=\tau(i,n)=
\begin{cases}
 i,  \text{ if}~~  i~~ \text{ is odd},\\ 
 i-1, \text{ if} ~~ i ~~\text{is even}.
\end{cases} 
\end{equation*}
Set $X=I \times J \times \mathbb{R}.$ For each $(i,j) \in \Sigma_m \times \Sigma_n,$ let $F_{ij}:X \rightarrow \mathbb{R}$ be a continuous function satisfying $$ F_{ij}(x_k,y_l,z_{kl})=z_{\tau(i,k),\tau(j,l)},  ~~~~~~ \forall ~~(k,l) \in \partial \Sigma_{m,0} \times  \partial \Sigma_{n,0}$$ and  $$|F_{ij}(x,y,z')- F_{ij}(x,y,z'')| \le \alpha_{ij} |z' - z''|,$$ for all $(x,y) \in I \times J$ and $z',z'' \in \mathbb{R},$ where $0 < \alpha_{ij} < 1$ is a given constant.\\
Now, for each $(i,j) \in \Sigma_m \times \Sigma_n,$ we define $W_{ij}:X \rightarrow I_i \times J_j \times \mathbb{R}$ by $$ W_{ij}(x,y,z)=\big(u_i(x),v_j(y),F_{ij}(x,y,z)\big).$$
Then $\{X, W_{ij}: (i,j) \in\Sigma_m \times \Sigma_n \}$ is an IFS. 
\begin{theorem} [\cite{Ruan}] 
	Let $\{X, W_{ij}: (i,j) \in\Sigma_m \times \Sigma_n \}$ be the IFS defined  above. Assume that $\{F_{ij}:(i,j) \in \Sigma_m \times \Sigma_n \}$ satisfies the following matching conditions:
	\begin{enumerate}
		\item for all $i \in \text{int}\Sigma_{m,0}, j \in \Sigma_n$ and $x^*=u^{-1}_i(x_i)=u^{-1}_{i+1}(x_i),$
		$$F_{ij}(x^*,y,z)=F_{i+1,j}(x^*,y,z),  \forall ~y \in J, z \in \mathbb{R}, $$ 
		\item  for all $i \in \Sigma_m, j \in \text{int} \Sigma_{n,0}$ and $y^*=v^{-1}_j(y_j)=v^{-1}_{j+1}(y_j),$ $$F_{ij}(x,y^*,z) = F_{i,j+1}(x,y^*,z),~ \forall~ x \in I, z \in \mathbb{R}.$$
	\end{enumerate}
	Then there exists a unique continuous function $g:I \times J \rightarrow \mathbb{R}  $ such that $g(x_i,y_j)=z_{ij}$ for all $(i,j) \in \Sigma_{m,0} \times \Sigma_{n,0}$ and $G= \bigcup_{(i,j) \in \Sigma_m \times \Sigma_n} W_{ij}(G),$ where $G$ is the  graph of $g$, that is,  $G=\big\{(x,y,g(x,y)):(x,y) \in I \times J \big\}$. 
\end{theorem}

\begin{definition}
The set $G$ appearing in the previous theorem is called  the FIS and function $g$ is termed  the bivariate FIF for the finite data set $\{(x_i,y_j,z_{ij}) \in \mathbb{R}^3 : i \in  \Sigma_{m,0}; j \in  \Sigma_{n,0}\}$.
\end{definition}
\section{Construction  of countable bivariate fractal interpolation surfaces}\label{sec3}
 Let $\mathbb{N}_0 = \mathbb{N}\bigcup\{0\}.$  First let us recall that to each double sequence $s: \mathbb{N}_0 \times  \mathbb{N}_0 \to \mathbb{R}$ denoted by $s(i,j) = s_{ij}$, there corresponds  three important limits, namely, 
$$ \lim_{i,j \to \infty} s_{ij},  ~~ \lim_{ i \to \infty} \big(\lim_{ j \to \infty} s_{ij}  \big)~~ \text{and} ~~  \lim_{j \to \infty} \big(\lim_{ i \to \infty} s_{ij}  \big).$$
Further, it is worth to recall that the existence of $\lim_{i,j \to \infty} s_{ij}$ does not ensure, in general,  the existence of the limits $\lim_{i \to \infty} s_{ij}$ for each fixed $j \in \mathbb{N}_0$ and $\lim_{j \to \infty} s_{ij}$ for each fixed $i \in \mathbb{N}_0$. 
\par 
A set $D= \{(x_i,y_j,z_{ij}): i,j \in \mathbb{N}_0\}\subset\mathbb{R}^3$ is said to be bivariate countable system data (CSD) if  $(x_i)_{i \in \mathbb{N}_0}$ and  $(y_j)_{j \in \mathbb{N}_0}$ are strictly increasing bounded above sequences and the double sequence $(z_{ij})$ is convergent in the sense that $\lim_{i,j\to\infty}z_{ij}$ exists. Let $ M=z_{ \infty \infty}= \lim_{i,j\to\infty}z_{ij}$. Further assume that  $\lim_{j\to\infty}z_{ij}<\infty$ for each fixed $i\in\mathbb{N}_0$ and $\lim_{i\to\infty}z_{ij}<\infty$ for each fixed $j\in\mathbb{N}_0$. Let $x_0 =a$, $\lim_{i\to\infty}x_i = x_\infty=b$, $y_0 = c$  and $\lim_{j\to\infty}y_j =y_\infty= d.$ Set $I = [a,b]$, $J = [c,d]$. Assume that $K$  is a sufficiently large compact interval containing  $\{z_{ij}:i,j=0,1,2,\ldots\}\cup\{M\}$  and $X = I\times J\times K$.   
\par This section is concerned with the existence of a continuous function $g:I\times J\rightarrow \mathbb{R}$ whose graph is an attractor of an appropriate  CIFS and
$$g(x_i,y_j) = z_{ij}\text{ for all } i,j\in \mathbb{N}_0.$$ Our analysis will be patterned after  Ruan's \cite{Ruan} elegant construction of fractal surface outlined in the previous section. 
\par 
For  $i,j\in\mathbb{N}$, let  $I_i = [x_{i-1},x_i]$ and $ J_j = [y_{j-1},y_j]$. Let $s = (0,1,0,1,\ldots),$ be a binary sequence with $0$ at odd places and $1$ at even places. Define $\tau:\mathbb{N}\times\{0,\infty\}\rightarrow\mathbb{N}$  by
 	\begin{equation}\label{cons0}
 	\begin{aligned}
 	\tau(i,0)= i-1 + s_i \text{ and }\tau(i,\infty)= i - s_i.
 	\end{aligned} 
 	\end{equation}

For $i,j\in\mathbb{N},$ let $u_i:I\rightarrow I_i$ and $v_j:I\rightarrow I_j$ be  contractive homeomorphisms satisfying
\begin{equation}\label{cons1}
    u_i(x_0) = x_{i-1+s_i},\hspace{0.5cm} u_i(x_{\infty}) = x_{i-s_i},
\end{equation}
\begin{equation}\label{cons2}
|u_i(x) -u_i(x')|\leq a_{i}|x-x'|\hspace{0.5cm}\text{ for all } x,x'\in I,
\end{equation}
\begin{equation}\label{cons3}
	v_j(y_0) = y_{j-1+s_j},\hspace{0.5cm} v_j(y_{\infty}) = y_{j-s_j},
\end{equation}
\begin{equation}\label{cons4}
|v_j(y)-v_j(y')|\leq b_{j}|y-y'|\hspace{0.5cm}\text{ for all } y,y'\in J,
\end{equation}
where $a_{i},b_{j}$ are positive constants such that $\|a\|_{\infty} := \sup_{i \in \mathbb{N}}a_{i}<1$ and $\|b\|_{\infty} = \sup_{j \in \mathbb{N}}b_{j}<1.$ Further assume that $\|\delta\|_{\infty} := \sup_{i,j}\delta_{ij} < 1$, where $\delta_{ij} = \max\{a_{i},b_{j}\}.$ Using  (\ref{cons1}) and (\ref{cons3}),  one can easily observe that
\begin{equation*}
\begin{aligned}
&u_i^{-1}(x_i) = u_{i+1}^{-1}(x_{i})\hspace{0.5cm}~ \forall~ i\in\mathbb{N}\quad \text{ and} \quad 
v_j^{-1}(y_j) = v_{j+1}^{-1}(y_{j})\hspace{0.5cm} ~\forall ~j\in\mathbb{N}.
\end{aligned}
\end{equation*}
For each $(i,j)\in\mathbb{N}\times\mathbb{N},\text{ let } F_{ij}:X\rightarrow K$ be continuous function satisfying
\begin{equation}\label{F_kl}
F_{ij}(x_k,y_l,z_{kl}) = z_{\tau(i,k),\tau(j,l)} \hspace{0.5cm}\forall~~ k,l\in\{0,\infty\},
\end{equation}
\begin{equation}\label{cons1F}
|F_{ij}(x,y,z) - F_{ij}(x',y',z)|\leq \theta\|(x,y) - (x',y')\|,
\end{equation}
where $\theta> 0$, $\|.\|$ is the Euclidean norm on $\mathbb{R}^2$ and 
\begin{equation}\label{cons2F}
|F_{ij}(x,y,z) - F_{ij}(x,y,z')|\leq \alpha_{ij}|z - z'|,
\end{equation}
where $\alpha_{ij}$ are positive constants such that $\|\alpha\|_{\infty} = \sup_{i,j}\alpha_{ij}<1$.
Now, for each $(i,j)\in\mathbb{N}\times\mathbb{N},$ we define $W_{ij}:X\rightarrow X$ by
\begin{equation}\label{CIFS}
W_{ij}(x,y,z) =\big (u_i(x),v_j(y),F_{ij}(x,y,z)\big).
\end{equation}
Then $\{X,W_{ij}:i,j\in\mathbb{N}\}$ is a countable IFS. Let $\mathcal{H}(X)$ denote the set of all non-empty compact subsets of $X$ endowed with the Hausdorff metric. Define the set-valued operator (see also the Hutchinson operator that made its debut in \cite{JH}) $\mathcal{W}:\mathcal{H}(X)\rightarrow\mathcal{H}(X)$ by\\
\begin{equation}
\mathcal{W}(B) = \overline{\bigcup_{i,j\geq 1}{W_{ij}(B)}}\quad \forall~~ B\in\mathcal{H}(X).
\end{equation}
By the definition we have
\begin{equation}
W_{ij}(x_k,y_l,z_{kl}) = \big(x_{\tau(i,k)},y_{\tau(j,l)},z_{\tau(i,k)\tau(k,l)}\big)~~\text{ for all } (k,l)\in\{0,\infty\}.
\end{equation}
For $(x,y,z),(x',y',z')\in\mathbb{R}^3$  and $ \delta>0,$ define a metric $d_{\delta}$ as follows
\begin{equation} \label{defmetric}
d_{\delta}((x,y,z),(x',y',z')) = \|(x,y) - (x',y')\| + \delta|z-z'|.
\end{equation}
It can be easily verified  that the metric $d_{\delta}$ defined above is equivalent to the Euclidean metric on $\mathbb{R}^3$ for all $\delta>0.$ The following result is analogous to Theorem \ref{Basthm1}. 
\begin{proposition}
 The countable iterated function system $\{X,W_{ij}:(i,j)\in\mathbb{N}\times\mathbb{N}\}$ is hyperbolic with respect to the metric $d_{\delta}$  defined in (\ref{defmetric}), with  $\delta := \inf_{i,j}\dfrac{1-2 \delta_{ij}}{2\theta}.$   Hence it possesses an attractor, that is, there exists a compact set $A\subseteq X$ such that $A=\mathcal{W}(A) = \overline{\bigcup_{i,j\geq 1}{W_{ij}(A)}}.$ 
\end{proposition}
\begin{proof}
	For $(x,y,z),(x',y',z')\in X,$ we have
	\begin{equation*}
	\begin{aligned}
	d_{\delta}\big(W_{ij}(x,y,z),W_{ij}(x',y',z')\big) &= \|(u_i(x),v_j(y)) - (u_i(x'),v_j(y'))\|\\&~ + \delta|F_{i,j}(x,y,z) - F_{i,j}(x',y',z')|\\
	\leq&~ \|(u_i(x),v_j(y)) - (u_i(x'),v_j(y'))\|\\& + \delta\big[|F_{i,j}(x,y,z) - F_{i,j}(x',y',z)|\\
	& + |F_{i,j}(x',y',z) - F_{i,j}(x',y',z')|\big]\\
	\leq & \delta_{ij}\|(x,y) - (x',y')\| + \delta\big[\theta\|(x,y) - (x',y')\| + \|\alpha\|_{\infty}|z-z'|\big]\\
	= & (\delta_{ij} + \delta \theta)\|(x,y) - (x',y')\| + \delta \|\alpha\|_{\infty}|z-z'|.
	\end{aligned}
	\end{equation*}
	By the choice of $\delta,$  we have $\sup_{i,j}(\delta_{ij} + \delta \theta) < 1.$ Hence,
	\begin{equation*}
	\begin{aligned}
	d_{\delta}\big(W_{ij}(x,y,z),W_{ij}(x',y',z')\big) \leq &~\max\{\sup_{i,j}(\delta_{ij} + \delta \theta),\|\alpha\|_{\infty}\}\\&~. \big(\|(x,y) - (x',y')\| + \delta|z-z'|\big)\\
	 =&~ \max\{\sup_{i,j}(\delta_{ij} + \delta \theta),\|\alpha\|_{\infty}\}\\&~d_{\delta}\big((x,y,z),(x',y',z')\big),
	\end{aligned}
	\end{equation*}
	completing the proof.
\end{proof}
The following theorem is obtained by modifying and adapting  results on the existence of various types of univariate and bivariate fractal interpolation functions scattered in the literature, see, for instance, \cite{MF1,Dalla,PM,RW,NAS1}. 
\begin{theorem}\label{CBFIF}
	Let $\{X,W_{ij}:(i,j)\in\mathbb{N}\times\mathbb{N}\}$ be the CIFS defined through  (\ref{cons0})- (\ref{CIFS}). Assume that for each $(i,j)\in\mathbb{N}\times\mathbb{N}$ the function $F_{ij}:X\rightarrow K$ further satisfy the following matching conditions
	\begin{enumerate}
		\item for all $i\in\mathbb{N}$ and $x^* = u_i^{-1}(x_i) = u_{i+1}^{-1}(x_{i}),$
		\begin{equation}\label{F1}
		F_{ij}(x^*,y,z) = F_{i+1,j}(x^*,y,z)\hspace{0.5cm} \forall~ y\in J, z\in K, \text{ and}
		\end{equation}
		\item for all $j\in\mathbb{N}$ and $y^* = v_j^{-1}(y_j) = v_{j+1}^{-1}(y_{j}),$
		\begin{equation}\label{F2}
		F_{ij}(x,y^*,z) = F_{i+1,j}(x,y^*,z)\hspace{0.5cm} \forall~ x\in I, z\in K.
		\end{equation}
	\end{enumerate} 
Then there exists a unique continuous function $g:I\times J\rightarrow K$ such that $g(x_i,y_j) = z_{ij}$ for all  $ i,j\in\mathbb{N}_0\times\mathbb{N}_0$ and $G=\{(x,y,g(x,y)):(x,y)\in I\times J\},$ the graph of $g,$ is the attractor of the CIFS defined above.
\end{theorem}
\begin{proof}
We shall denote by $\mathcal{C}(I \times J)$,  the set of all continuous real-valued functions defined on $I \times J$, 	and endow it with the sup-norm. Consider the set $$\mathcal{C}^*(I\times J) =\big\{g\in \mathcal{C}(I\times J): g(x_k,y_l) = z_{kl} \text{ for all }k,l\in\{0,\infty\}\big \}$$
endowed with the uniform metric. Bearing in mind that   $(x_i)_{ i \in \mathbb{N}_0} $ and $(y_j)_{ j \in \mathbb{N}_0} $ are strictly  increasing sequences 
satisfying  $\lim_{ i \to \infty} x_i = x_\infty$ and $\lim_{ j \to \infty} y_j= y_\infty$  we  define $T: \mathcal{C}^*(I\times J)\rightarrow \mathcal{C}^*(I\times J)$ as follows.
For $h \in \mathcal{C}^*(I\times J)$ and $(x,y) \in I\times J$ 
	\begin{equation}\label{T}
	\begin{aligned}
	T(h)(x,y)~:=~
	\begin{cases}
	F_{ij}\big(u_{i}^{-1}(x),v_{j}^{-1}(y),h(u_{i}^{-1}(x),v_{j}^{-1}(y))\big),&~\\ \text{if}~(x,y)\in I_{i}\times J_{j}~\text{for some}~~(i,j)\in \mathbb{N}\times\mathbb{N};\\
	\displaystyle\lim_{j\to \infty}F_{ij}\big(u_{i}^{-1}(x),v_{j}^{-1}({y}_j),h(u_{i}^{-1}(x),v_{j}^{-1}(y_j))\big),~\\ \text{if}~x\in I_{i} ~\text{for some}~i\in \mathbb{N}\text{ and } y = y_{\infty,};\\
	\displaystyle\lim_{i\to \infty}F_{ij}\big(u_{i}^{-1}(x_i),v_{j}^{-1}(y),h(u_{i}^{-1}(x_i),v_{j}^{-1}(y))\big)\\~ \text{if}~x = x_{\infty}\text{ and } y \in J_{j} ~\text{for some}~  j\in\mathbb{N};\\
	z_{\infty\infty},~ \\\text{if}~ x = x_{\infty}\text{ and } y = y_{\infty}.
	\end{cases}
	\end{aligned}
	\end{equation}
  It follows from (\ref{F1}) and (\ref{F2}) that $T(h)$ is well-defined on the boundary of $I_{i}\times J_{j}$ for all $(i,j)\in\mathbb{N}\times\mathbb{N}$. For instance,  let us note the following.  Let  $y_{j-1} \le y \le y_j$. Then $(x_{i-1},y) \in (I_i  \times J_j)  \cap(I_{i-1} \times J_j)$.  Treating $(x_{i-1},y)$ as an element in $I_{i-1}\times J_j$ we have 
\begin{equation}\label{consnew1}
T(h) (x_{i-1},y)= F_{i-1,j}\big(u_{i-1}^{-1}(x_{i-1}),v_{j}^{-1}(y),h(u_{i-1}^{-1}(x_{i-1}),v_{j}^{-1}(y))\big).
\end{equation}
On the other hand, taking $(x_{i-1},y)$ as an element in $I_ i\times J_j$ we get
\begin{equation}\label{consnew2}
T(h) (x_{i-1},y)= F_{ij}\big(u_i^{-1}(x_{i-1}),v_{j}^{-1}(y),h(u_i^{-1}(x_{i-1}),v_{j}^{-1}(y))\big).
\end{equation}
Noting 
\begin{equation*}
	\begin{aligned}
	u_{i-1}^{-1}(x_{i-1})=u_i^{-1}(x_{i-1})=
	\begin{cases}
x_0, ~~~~ \text{if}~i~~~ \text{is odd}\\
x_\infty,~~~~~ \text{if}~i~~~\text{is even}
\end{cases}
\end{aligned}
\end{equation*}
via (\ref{F1}) it follows from (\ref{consnew1}) and  (\ref{consnew2})  that $T(h)(x_{i-1},y)$ is uniquely determined. Similarly, one can prove that $T(h)$  is continuous on $[a,b)\times [c,d).$  
The continuity of $T(h)$ on $[a,b)\times [c,d)$, together with the facts that the increasing sequences  $x_i \to x_\infty$, $y_j \to y_\infty$  and the double sequence  $z_{ij}\rightarrow M$,  yield the continuity of $T(h)$  on $I\times J. $ 
   \par For any $(i,j)\in\mathbb{N}\times\mathbb{N}$, by the condition on $\tau$ given in (\ref{cons0}), we can choose $(k,l)\in\{0,\infty\}\times\{0,\infty\}$ such that $i=\tau(i,k)$ and $j=\tau(j,l).$ By (\ref{cons1}) it follows that $x_k = u_{i}^{-1}(x_i)$ and  $ y_l = v_{j}^{-1}(y_j).$ Using (\ref{F_kl}) and (\ref{T}) 
        $$T(h)(x_i,y_j) = F_{ij}\big(x_k,y_l,h(x_k,y_l)\big) = F_{ij}(x_k,y_l,z_{kl}) = z_{\tau(i,k)\tau(j,l)} = z_{ij}.$$
Continuity of $T(h)$ further implies $T(h)(x_k,y_l) = z_{kl}$ for all $k,l\in\{0,\infty\}.$ Consequently,  $T$ maps $\mathcal{C}^*(I\times J)$ into itself. 
\par Let  $h_1,h_2\in \mathcal{C}^*(I\times J)$. For $(x,y)\in I_{i}\times J_{j}$, where $(i,j)\in\mathbb{N}\times\mathbb{N}$ 
	\begin{equation*}
	\begin{aligned}
	|T(h_1)(x,y) - T(h_2)(x,y)| &= \Big|F_{ij}\big(u_{i}^{-1}(x),v_{j}^{-1}(y),h_1(u_{i}^{-1}(x),v_{j}^{-1}(y))\big)\\
	                        & - F_{ij}\big(u_{i}^{-1}(x),v_{j}^{-1}(y),h_2(u_{i}^{-1}(x),v_{j}^{-1}(y))\big)\Big|\\
	                        &\leq \alpha_{ij}\big|h_1(u_{i}^{-1}(x),v_{j}^{-1}(y)) - h_2(u_{i}^{-1}(x),v_{j}^{-1}(y))\big|\\
	                        &\leq \|\alpha\|_{\infty}\|h_1 - h_2\|_{\infty}.
	\end{aligned}
	\end{equation*}
	Similarly, for all other $ (x,y)$ in $I \times J$. Therefore
	$$\|T(h_1) -T( h_2)\|_{\infty} \leq \|\alpha\|_{\infty}\|h_1 - h_2\|_{\infty}.$$
	That is, $T$ is a contraction map, and  hence by the Banach fixed point theorem it follows that there exists a unique $g\in \mathcal{C}^*(I\times J)$ satisfying 
\begin{equation}\label{self}
	\begin{aligned}
	g(x,y)~:=~
	\begin{cases}
	F_{ij}\big(u_{i}^{-1}(x),v_{j}^{-1}(y),g(u_{i}^{-1}(x),v_{j}^{-1}(y))\big),\\ \text{if}~(x,y)\in I_{i}\times J_{j}~\text{for some}~~(i,j)\in \mathbb{N}\times\mathbb{N};\\
	\displaystyle\lim_{ j\to \infty}F_{ij}\big(u_{i}^{-1}(x),v_{j}^{-1}(y_j),g(u_{i}^{-1}(x),v_{j}^{-1}(y_j))\big), \\\text{if}~x\in I_{i} ~\text{for some}~i\in \mathbb{N}\text{ and } y = y_{\infty;}\\
	\displaystyle\lim_{i\to \infty}F_{ij}(u_{i}^{-1}(x_i),v_{j}^{-1}(y),g(u_{i}^{-1}(x_i),v_{j}^{-1}(y))),\\ \text{if}~x = x_{\infty}\text{ and } y \in J_{j} ~\text{for some}~  j\in\mathbb{N};\\
	z_{\infty\infty}, \\ \text{if}~ x = x_{\infty}\text{ and } y = y_{\infty}.
	\end{cases}
	\end{aligned}
	\end{equation}
Now let $G = \{(x,y,g(x,y)):~ (x,y)\in I\times J\}$ be the graph of $g.$ For $(x,y)\in  [a,b)\times [c,d)$, there exists $ i_0,j_0\in \mathbb{N}\times\mathbb{N}$ such that $(x,y)\in I_{i_0}\times J_{j_0}.$ Thus by (\ref{self}),
	\begin{equation*}
	\begin{aligned}
	\big(x,y,g(x,y)\big) &=\big (x,y, F_{i_0j_0}(u_{i_0}^{-1}(x),v_{j_0}^{-1}(y),g(u_{i_0}^{-1}(x),v_{j_0}^{-1}(y)))\big)\\
	             &= W_{i_0j_0}\big(u_{i_0}^{-1}(x),v_{j_0}^{-1}(y),g(u_{i_0}^{-1}(x),v_{j_0}^{-1}(y))\big)\\ & \in W_{i_0j_0}(G).
	\end{aligned}
	\end{equation*}
	For $x\in [a,b)$, there exists $i_0\in \mathbb{N}$  such that $x \in I_{i_o}$. So,
	\begin{equation*}
	\begin{aligned}
	\big(x,y_\infty,g(x,y_\infty)\big) &= \big(x,\lim_{j\to\infty}y_j, \displaystyle\lim_{j\to \infty}F_{i_0j}(u_{i_0}^{-1}(x),v_{j}^{-1}(y_j),g(u_{i_0}^{-1}(x),v_{j}^{-1}(y_j)))\big)\\
	&=\lim_{j\to\infty} W_{i_0j}\big(u_{i_0}^{-1}(x),v_{j}^{-1}(y_j),g(u_{i_0}^{-1}(x),v_{j}^{-1}(y_j))\big) \\ & \in \overline{\bigcup_{j\geq 1}{W_{i_0j}(G)}}\\& \subset\overline{\bigcup_{i,j\geq 1}{W_{ij}(G)}}.
	\end{aligned}
	\end{equation*}
	Similarly, for any $y\in[c,d)$, we have  $\big(x_\infty,y,g(x_\infty,y)\big)\in\overline{\bigcup_{i,j\geq 1}{W_{ij}(G)}}$  and  $$\big(x_\infty,y_\infty,g(x_\infty,y_\infty)\big)\in\overline{\bigcup_{i,j\geq 1}{W_{ij}(G)}}.$$ Combining these facts we infer that 
	$$G\subset\overline{\bigcup_{i,j\geq 1}{W_{ij}(G)}}.$$
	Conversely, for $(x,y)\in I \times J$ and $(i,j)\in\mathbb{N}\times\mathbb{N}$, we have $u_i(x)\in I_i$ and $v_j(y)\in J_j.$ Therefore
	\begin{equation*}
	\begin{aligned}
	W_{ij}\big(x,y,g(x,y)\big) &= \big (u_i(x),v_j(y),F_{ij}(x,y,g(x,y))\big)\\
	&=  \big(u_i(x),v_j(y),g(u_i(x),v_j(y))\big) \\& \in G.    
	\end{aligned}
	\end{equation*}
	Since $G$ is closed we get
	$$\overline{\bigcup_{i,j\geq 1}{W_{ij}(G)}}\subset G.$$
	Hence,
	$$G = \overline{\bigcup_{i,j\geq 1}{W_{ij}(G)}},$$
	completing the proof.
\end{proof}
\begin{definition}
We refer to the set $G$ in the above theorem as the countable fractal interpolation surface (CFIS) and the bivariate function $g$ as the countable fractal interpolation function (CFIF) associated with the given data.
\end{definition}
\subsection{Some Approximations to Countable FIS}
Here we demonstrate some techniques to approximate the aforementioned countable bivariate fractal interpolation function. 
\begin{theorem}
	Let $g_0\in \mathcal{C}^*(I\times J)$ be arbitrary and $G(g_0)$ be the graph of $g_0$. Then the graph of $T(g_0)$ is $\mathcal{W}\big(G(g_0)\big).$ Moreover, the graph of $T^n(g_0)$ is $\mathcal{W}^n\big(G(g_0)\big)$ for any $n\in\mathbb{N}.$ Consequently,  if $G\big(T^n(g_0)\big)$ denotes the graph of $T^n(g_0),$ then the sequence $\big(G\big(T^n(g_0)\big)\big)_{n\in\mathbb{N}}$ converges to $G(g_0)$ with respect to the Hausdorff metric.
\end{theorem}
\begin{proof} Note that
	\begin{equation*}
	\begin{aligned}
	&(x',y',z')\in\bigcup_{i,j\geq 1}{W_{ij}\big(G(g_0)\big)}\\
	&\Rightarrow (x',y',z')\in W_{ij}\big(G(g_0)\big)\text{ for some } (i,j)\in\mathbb{N}\times\mathbb{N}\\
	&\Rightarrow (x',y',z') = W_{ij}\big(x,y,g_0(x,y)\big)\text{ for some } (x,y)\in I_i\times J_j\\
	&\Rightarrow (x',y',z') = \big(u_i(x),v_j(y),F_{ij}(x,y,g_0(x,y))\big)\\
	&\Rightarrow (x',y',z') = \big(u_i(x),v_j(y),T(g_0)(u_i(x),v_j(y))\big) \in G(T(g_0)). 
	\end{aligned}
	\end{equation*}
	Thus,
	$$\mathcal{W}\big(G(g_0)\big) = \overline{\bigcup_{i,j\geq 1}{W_{ij}(G(g_0))}}\subset G(T(g_0)).$$
	Conversely, let $\big(x,y,T(g_0)(x,y)\big)\in G(T(g_0))$ for some $(x,y)\in [a,b)\times [c,d).$ Then there exists some $(i,j)\in\mathbb{N}\times\mathbb{N}$ such that $(x,y)\in I_i\times J_j,$  and hence there exists $ (x',y')\in I\times J$ such that $(x,y) = \big(u_i(x'),v_j(y')\big).$ Therefore
	\begin{equation*}
	\begin{aligned}
	\big(x,y,T(g_0)(x,y)\big) &= \big(u_i(x'),v_j(y'),T(g_0)(u_i(x'),v_j(y'))\big)\\
	                  &= \big(u_i(x'),v_j(y'),F_{ij}(x',y,'g_0(x'y'))\big)\\
	                  &= W_{ij}\big(x',y',g_0(x'y')\big)\in W_{ij}(G(g_0)).
	\end{aligned}
	\end{equation*}
	Using the above observation,  as in the last theorem, we get $ G\big(T(g_0)\big)\subset\mathcal{W}\big(G(g_0)\big).$ Combining these
	$$ G\big(T(g_0)\big)=\mathcal{W}\big(G(g_0)\big).$$
	Iterating the above steps we get
	$$ G(T^n(f_0))=\mathcal{W}^n\big(G(g_0)\big) \text{ for any } n\in \mathbb{N}.$$
	Now, from the Banach fixed point theorem we know that for any set $A\in\mathcal{H}(X),$ the sequence $(\mathcal{W}^n(A))_{n\in\mathbb{N}}$ converges to the fixed point $G$ with respect to the Hausdorff metric. Taking $A = G(g_0)$ proves the result.
\end{proof}
  The following theorem,  which is in the setting of  a double sequence of compact sets, can be proved verbatim to Theorem 1.1 in \cite{NAS1}. 
 \begin{theorem}\label{thmnewver0}
 	Let $(X,d)$ be a complete metric space and $(K_{m,n})_{m,n\in\mathbb{N}}$ be a double sequence of compact subsets of $X.$ 
 	\begin{enumerate}
 		\item  Assume that $K_{m,n}\subset K_{m+1,n}$, $K_{m,n}\subset K_{m,n+1}$ and  $K_{m,n}\subset K_{m+1,n+1}$ for all $m,n\in\mathbb{N}.$ Further let the set $K:= \bigcup_{m,n\geq 1}{K_{m,n}}$ is relatively compact. Then
 		$$\overline{K} = \overline{\bigcup_{m,n\geq 1}{K_{m,n}}} = \lim_{m,n \to \infty} K_{m,n}.$$
 		\item If $K_{m+1,n}\subset K_{m,n}$, $K_{m,n+1}\subset K_{m,n}$  and  $K_{m+1,n+1}\subset K_{m,n}$ for all $m,n\in\mathbb{N},$ then
 		$$C:= \cap_{m,n\geq 1}{K_{m,n}} = \lim_{m,n} K_{m,n}.$$
 	\end{enumerate}
\end{theorem}
 Let $(X,d)$ be a compact metric space.  For $(i,j)\in\mathbb{N}\times\mathbb{N}$, let $\omega_{ij}:X\rightarrow X$ be contraction maps with contractivity factor $r_{ij}$ such that $\sup_{i,j}r_{ij} < 1.$  Consider the CIFS $\{X,\omega_{ij}:~i\in\mathbb{N},j\in\mathbb{N}\}.$
 For $r \in N$, we denote by $\Sigma_r$, the segment $\{1,2,\dots,r\}$ of $\mathbb{N}$.  Assume  $m,n\in\mathbb{N}$.  Let us call $\{X,\omega_{ij}:~i\in\mathbb{N}_m, j\in\mathbb{N}_n\}$ as the \textit{partial IFS} associated with the CIFS $\{X,\omega_{ij}:~i\in\mathbb{N},j\in\mathbb{N}\}.$ We shall denote the attractor of the aforesaid partial IFS as $A_{m,n}.$ Our aim is to  prove that the double sequence $(A_{m,n})_{m,n\in\mathbb{N}}$ provides an approximation of the attractor of the aforementioned CIFS.   \begin{lemma}\label{inc}
 	Let $m,n\in\mathbb{N},$ then we have $A_{m,n}\subset A_{m+1,n},A_{m,n}\subset A_{m,n+1}$ and $A_{m,n}\subset A_{m+1,n+1}$
 \end{lemma}
\begin{proof}
	Let $i_1,i_2,\ldots, i_m\in\{1,2,\ldots, m\},~j_1,j_2,\ldots, j_n\in\{1,2,\ldots, n\}.$ Define functions $\omega_{i_1i_2\ldots i_m,j_1j_2\ldots j_n} :X\rightarrow X$ by
	\begin{equation*}
	\omega_{i_1i_2\ldots i_m,j_1j_2\ldots j_n}(x) = \omega_{i_1j_1}\circ\ldots\circ \omega_{i_1j_n}\circ \omega_{i_2j_1}\circ\ldots\circ \omega_{i_2j_n}\circ\ldots \circ \omega_{i_mj_1}\circ\ldots\circ \omega_{i_mj_m}(x).
	\end{equation*}
	Obviously $\omega_{i_1i_2\ldots i_m,j_1j_2\ldots j_n}$ is a contraction map with contractivity factor at most  $\prod_{k=1}^{k=m}\prod_{l=1}^{l=n}r_{{i_k}{j_l}}<1.$ Let $a_{i_1i_2\ldots i_m,j_1j_2\ldots j_n}$ be its unique fixed point, then following \cite{JH} we have\begin{equation*}
	A_{m,n} = \overline{\{a_{i_1i_2\ldots i_m,j_1j_2\ldots j_n}: i_1,i_2,\ldots, i_m\in\mathbb{N}_m,~j_1,j_2,\ldots, j_n\in\mathbb{N}_n\}}.
	\end{equation*}
	The proof follows immediately from the above observation.
\end{proof}

\begin{lemma}[Lemma 2.2, \cite{NAS1}]\label{family}
	If $(E_{\lambda})_{\lambda\in\Lambda}$ is a family of subsets  of a topological space, then
	\begin{equation*}
	\overline{\bigcup_{\lambda\in\Lambda}\overline{E_{\lambda}}} = \overline{\bigcup_{\lambda\in\Lambda}E_{\lambda}}.
	\end{equation*}
\end{lemma}
The following theorem is analogous to a result in  \cite{NAS1}. 
\begin{theorem}
The set $A  = \overline{\bigcup_{m,n\geq 1}{A_{m,n}}}\in\mathcal{H}(X)$ is the attractor of the CIFS $\{X,\omega_{ij}:~i\in\mathbb{N},j\in\mathbb{N}\}.$
\end{theorem}
\begin{proof}
	We have
	\begin{equation}\label{set1}
	\begin{aligned}
	\overline{\bigcup_{i,j\geq 1}\omega_{ij}(A_{m,n})} &= \overline{\bigcup_{i=1}^{m}\bigcup_{j=1}^{n}\omega_{ij}(A_{m,n})} \bigcup \overline{\bigcup_{i=m+1}^{\infty}\bigcup_{i=n+1}^{\infty}\omega_{ij}(A_{m,n})}\\
	&= A_{m,n}\bigcup \overline{\bigcup_{i=m+1}^{\infty}\bigcup_{i=n+1}^{\infty}\omega_{ij}(A_{m,n})}.
	\end{aligned}
	\end{equation}
	By Lemma \ref{family} we have\\
	\begin{equation}\label{set2}
	\begin{aligned}
	\overline{\bigcup_{m,n\geq 1}\bigcup_{i,j\geq 1}\omega_{ij}(A_{m,n})} = \overline{\bigcup_{m,n\geq 1}\overline{\bigcup_{i,j\geq 1}\omega_{ij}(A_{m,n})}}
	\end{aligned}
	\end{equation}
	Thus, using (\ref{set1}) and (\ref{set2}) we have
	\begin{equation}\label{set3}
	\begin{aligned}
	\overline{\bigcup_{i,j\geq 1}\bigcup_{m,n\geq 1}\omega_{ij}(A_{m,n})} &= \overline{\bigcup_{m,n\geq 1}\bigcup_{i,j\geq 1}\omega_{ij}(A_{m,n})}\\
	 &= \overline{\bigcup_{m,n\geq 1}\overline{\bigcup_{i,j\geq 1}\omega_{ij}(A_{m,n})}}\\
	 &= \overline{\bigcup_{m,n\geq 1}\Big(A_{m,n}\bigcup\overline{\bigcup_{i\geq m+1}\bigcup_{j\geq n+1}\omega_{ij}(A_{m,n})}}\Big)\\
	 &= \overline{\bigcup_{m,n\geq 1}A_{m,n}}\bigcup\overline{\bigcup_{m,n\geq 1}\bigcup_{i\geq m+1}\bigcup_{j\geq n+1}\omega_{ij}(A_{m,n})}\\
	\end{aligned}
	\end{equation}
	By Lemma \ref{inc} we have\\
	\begin{equation}\label{set4}
	\overline{\bigcup_{m,n\geq 1}\bigcup_{i\geq m+1}\bigcup_{j\geq n+1}\omega_{ij}(A_{m,n})}\subset\overline{\bigcup_{m,n\geq 1}A_{m,n}}.
	\end{equation}
	Combining (\ref{set3})  and (\ref{set4}) we get
	\begin{equation}\label{set5}
	\overline{\bigcup_{i,j\geq 1}\bigcup_{m,n\geq 1}\omega_{ij}(A_{m,n})} = \overline{\bigcup_{m,n\geq 1}A_{m,n}}.
	\end{equation}
	Now, using the continuity of $\omega_{ij}$ for each $i,j\in\mathbb{N},$  (\ref{set5}) and Lemma \ref{family} we have
	\begin{equation*}
	\begin{aligned}
	\overline{\bigcup_{i,j\geq 1}\omega_{ij}\Big(\overline{\bigcup_{m,n\geq 1}{A_{m,n}}}\Big)} &\subset \overline{\bigcup_{i,j\geq 1}\overline{\omega_{ij}\Big(\bigcup_{m,n\geq 1}{A_{m,n}}\Big)}}\\
	&= \overline{\bigcup_{i,j\geq 1}\omega_{ij}\Big(\bigcup_{m,n\geq 1}{A_{m,n}}\Big)}\\
	 &= \overline{\bigcup_{i,j\geq 1}\bigcup_{m,n\geq 1}\omega_{ij}(A_{m,n})}= \overline{\bigcup_{m,n\geq 1}{A_{m,n}}}\\
	\end{aligned}
	\end{equation*}
	Thus, 
	\begin{equation}\label{set6}
	\overline{\bigcup_{i,j\geq 1}\omega_{ij}\Big(\overline{\bigcup_{m,n\geq 1}{A_{m,n}}}\Big)}\subset \overline{\bigcup_{m,n\geq 1}{A_{m,n}}}.
	\end{equation}
	Conversely using (\ref{set5}) we have
	\begin{equation*}
	\begin{aligned}
	\overline{\bigcup_{m,n\geq 1}{A_{m,n}}} &= \overline{\bigcup_{i,j\geq 1}\bigcup_{m,n\geq 1}\omega_{ij}(A_{m,n})}\\
	&= \overline{\bigcup_{i,j\geq 1}\omega_{ij}\Big(\bigcup_{m,n\geq 1}{A_{m,n}}\Big)}\\
	&\subset \overline{\bigcup_{i,j\geq 1}\omega_{ij}\Big(\overline{\bigcup_{m,n\geq 1}{A_{m,n}}}\Big)}\\
	\end{aligned}
	\end{equation*}
	That is, 
	\begin{equation}\label{set7}
	\overline{\bigcup_{m,n\geq 1}{A_{m,n}}}\subset \overline{\bigcup_{i,j\geq 1}\omega_{ij}\Big(\overline{\bigcup_{m,n\geq 1}{A_{m,n}}}\Big)}
	\end{equation}
Equations (\ref{set6}) and (\ref{set7}) together yield the following
	\begin{equation*}
	\overline{\bigcup_{i,j\geq 1}\omega_{ij}\Big(\overline{\bigcup_{m,n\geq 1}{A_{m,n}}}\Big)} = \overline{\bigcup_{m,n\geq 1}{A_{m,n}}},
	\end{equation*}
	completing the proof.
\end{proof}	
The above theorem in conjunction with Theorem \ref{thmnewver0} and Lemma \ref{inc} provides the promised approximation of the attractor of CIFS by the attractors of the partial IFSs. To be precise, 
\begin{theorem}
	Let  $A$ be the attractor of the CIFS $\{X,\omega_{ij}:~i\in\mathbb{N},j\in\mathbb{N}\}$  and  $(A_{m,n})$ be the double sequence of attractors of the associated  partial IFSs $\{X,\omega_{ij}:~i\in\mathbb{N}_m, j\in\mathbb{N}_n\}$.  Then 
	$$A = \lim_{m,n}A_{m,n},$$
	where the  limit is taken with respect to the Hausdorff metric.
\end{theorem}
As a special case of the previous theorem, we have the following.
\begin{corollary}
	For $m,n\in\mathbb{N}$ consider the partial IFS $\{X,W_{ij}:~i\in\Sigma_m,j\in\Sigma_n\}$ associated with the CIFS $\{X,W_{ij}:~i\in\mathbb{N},j\in\mathbb{N}\},$ (Cf. (\ref{CIFS})) defined in the construction of our countable FIS. Let $G_{m,n}$ be the attractor of $\{X,W_{ij}:~i\in\mathbb{N}_m,j\in\mathbb{N}_n\}$ and $G$ be the CFIS obtained in Theorem \ref{cons1}. Then we have
	$$\lim_{m,n} G_{m,n} = G,$$
	where the limit is taken with respect to Hausdorff metric.
\end{corollary}
\section{A Parameterized Family of Bivariate Fractal Functions and Associated  Fractal Operator}
As mentioned in the introduction, with an intension to explore some approximation theoretic aspects,  we consider here a special case of the countatable  bivariate  FIF constructed in the previous section. 
\begin{definition}
	Let $I \times J = [a,b]  \times [c,d] \subset \mathbb{R}^2$. We say that  $\Delta = \{x_i: i=0,1,2,\ldots\} \times \{y_j: j=0,1,2, \ldots\} \subset I \times J$ is  a partition of $I \times J$ if the sequences $(x_i)_{i\geq 0}$ and  $(y_j)_{j\geq 0}$ are such that $x_0=a$, $y_0=c$, both are strictly increasing, and converge to $b$ and $d$ respectively.
\end{definition}
	Let   $ \text{Lip}(I\times J) \subset \mathcal{C} (I\times J) $ denote the set of all Lipschitz continuous real-valued functions defined on  $I\times J $. That is,  $f \in \text{Lip}(I\times J) $, if there exists a constant $l_f>0$ such that 
$$ |f(x,y) - f(s,t)| \le l_f \| (x,y) - (s,t)\| ~~\forall ~~(x,y), (s,t) \in I \times J.$$
Let $f \in \text{Lip}(I\times J),$ be an arbitrary but fixed function, which we refer to as the germ  function or seed function. Consider the  countable data set  $D = \{(x_i,y_j,f(x_i,y_j)): i,j \in \mathbb{N}_0\}.$

\begin{definition}
	A scale function $\alpha = (\alpha_{ij})_{i,j\in\mathbb{N}}$ is a double sequence of Lipschitz functions on $I \times J$ such that 
	\begin{enumerate}
		\item $\sup_{i,j} |\alpha_{ij}| < \infty,$ where $|\alpha_{ij}|$ denotes the Lipschitz constant of $\alpha_{ij}.$
		\item  $\|\alpha\|_{\infty} := \sup_{i,j} \|\alpha_{ij}\|_{\infty} < 1,$ where $\|\alpha_{ij}\|_{\infty}$ denotes the sup-norm of $\alpha_{ij}.$
	\end{enumerate}
\end{definition}
We construct a class of bivariate FIFs corresponding to the countable data  $D = \{(x_i,y_j,f(x_i,y_j)): i,j \in \mathbb{N}_0\}$  by choosing appropriate maps $u_i$, $v_j$ and $F_{ij}$  in the countable IFS defined in the previous section. 
For $i \in \mathbb{N}$, let us take the function $u_i: I \to I_i$ as
 $$u_i(x) = a_ix + c_i,$$ 
 where $a_i$ and $c_i$ are chosen such that (\ref{cons1}) is satisfied. Similarly, for $j \in \mathbb{N}$ define $v_j:J \to J_j$ by 
 $$v_j(y) = b_jy + d_j,$$
 where $b_j$ and $d_j$ are chosen such that (\ref{cons3}) is satisfied. Assume that  $L: \text{Lip}(I\times J) \to \text{Lip}(I\times J)$ is an operator satisfying the boundary conditions
\begin{equation}
L(f)(x_k,y_l) = f(x_k,y_l) \quad \text{ for all } \quad k,l \in \{0,\infty\}.
\end{equation}
Finally,  let us take the function $F_{ij}: I \times J \times K \to K$ as
 \begin{equation}\label{F_{ij}}
 F_{ij}(x,y,z) = \alpha_{ij}(u_i(x),v_j(y))z + f(u_i(x),v_j(y)) - \alpha_{ij}(u_i(x),v_j(y))L(f)(x,y).
 \end{equation}
Some routine computations will show that  the functions $ F_{ij}$  in (\ref{F_{ij}}) satisfy the conditions (\ref{F_kl})-(\ref{cons2F}) and the matching conditions (\ref{F1})-(\ref{F2}) specified in Theorem \ref{CBFIF}. Ergo, by Theorem \ref{CBFIF} there exists a bivariate fractal interpolation function  interpolating the countable data set $\{(x_i,y_j,f(x_i,y_j)): i,j \in \mathbb{N}_0\}$ obtained by sampling the germ function $f$ at countable number of points in $I \times J$.  We denote this bivariate fractal interpolation function corresponding to the germ function $f$ by  $f^{\alpha}_{\Delta,L}$.
\begin{definition}
The aforementioned function  $f^{\alpha}_{\Delta,L}$ is termed as the (countable bivariate) $\alpha$-fractal function associated to the germ function $f$ with respect to the parameters $\alpha,~\Delta$ and $L.$
In view of  Theorem \ref{CBFIF} it is evident that $f^{\alpha}_{\Delta,L}$  satisfies 
\begin{equation}\label{FPeq}
	f^{\alpha}_{\Delta,L}(x,y) =
	\begin{aligned}
	\begin{cases}
	f(x,y) + \alpha_{ij}(x,y)\big[f^{\alpha}_{\Delta,L}(u_{i}^{-1}(x),v_{j}^{-1}(y)) - L(f)(u_{i}^{-1}(x),v_{j}^{-1}(y))\big]\\ ~\text{if}~(x,y)\in I_{i}\times J_{j};\\
	
	f(x,d) + \alpha_{ij}(x,d)\displaystyle\lim_{j \to \infty}\big[f^{\alpha}_{\Delta,L}(u_{i}^{-1}(x),v_{j}^{-1}(y_j)) - L(f)(u_{i}^{-1}(x),v_{j}^{-1}(y_j))\big]\\ \text{if}~~x\in I_{i},~y = d;\\
	
	f(b,y) + \alpha_{ij}(b,y)\displaystyle\lim_{i \to  \infty}\big[f^{\alpha}_{\Delta,L}(u_{i}^{-1}(x_i),v_{j}^{-1}(y)) - L(f)(u_{i}^{-1}(x_i),v_{j}^{-1}(y))\big],\\\text{if}~x = b,~y \in J_{j};\\
	
	f(b,d),\\ \text{if}~~x = b \text{ and } y =d.
	
	\end{cases}.
	\end{aligned}
	\end{equation}
\end{definition}
\begin{remark}
In fact, $\{f^{\alpha}_{\Delta,L}\}$ provides a family of fractal functions, obtained for different choices of parameters $\alpha$, $\Delta$ and $L$,  corresponding to the germ function $f$. Note that 

$$ f^{\alpha}_{\Delta,L} (x_i, y_j) = f(x_i , y_j) ~~\forall ~~i,j\in \mathbb{N}_0.$$

\end{remark} 
\begin{remark}
The fundamental impetus for the definition of parameterized family of fractal functions reported above was Navascues's construction of $\alpha$-fractal function widely recognized in the approximation theory of univariate fractal functions, see,  for instance,  \cite{MN1, MN2}. The difference here is that the germ function is bivariate and that sampling of the germ function is done at a countable number of points. 

\end{remark} 

\begin{definition}
 Let $\alpha$, $\Delta$ and $L$	be fixed. The operator  $\mathcal{F}^{\alpha}_{\Delta,L}:\text{Lip}(I \times J) \subset \mathcal{C}(I \times J)\to \mathcal{C}(I \times J)$ that assigns to  each  $f \in  \text{Lip}(I \times J)$  its fractal counterpart 
$ f^{\alpha}_{\Delta,L}$ is called the $\alpha$-fractal operator on Lip($I \times J$). That is,
$\mathcal{F}^{\alpha}_{\Delta,L}(f) :=  f^{\alpha}_{\Delta,L}$. 

\end{definition}

\begin{proposition}\label{prop5.7}
	Let $f\in \text{Lip}(I\times J)$ be the germ function and  $f_{\Delta,L}^{\alpha}$ be the $\alpha$-fractal function associated with $f$ corresponding to  the partition $\Delta$, scale function $\alpha$, and parameter map $L$. Then we have the following inequality.
	$$\| f_{\Delta,L}^{\alpha} - f\|_{\infty} \leq \|\alpha\|_{\infty}\| f_{\Delta,L}^{\alpha} - L(f)\|_{\infty}.$$
\end{proposition}
\begin{proof}
	The proof follows from the self-referential equation (\ref{FPeq}) satisfied by $f_{\Delta,L}^{\alpha}$ and some routine calculations; see also \cite{MN1}.
\end{proof}	
\begin{corollary}\label{cor1}
	Let $f \in \text{Lip}(I \times J)$ be the germ function and $f^{\alpha}_{\Delta,L}$ be the $\alpha$-fractal function  associated with $f$ with respect to the parameters $\alpha$, $\Delta$ and $L$. Then we have the following inequality
	$$\| f_{\Delta,L}^{\alpha}-L(f)\|_{\infty} \leq \frac{1}{1-\|\alpha\|_{\infty}}\|f-L(f)\|_{\infty}.$$	
	In particular, if $L = Id,$ the identity operator on $\text{Lip}(I \times J)$, then $\mathcal{F}_{\Delta,L}^{\alpha} = Id.$
\end{corollary}
\begin{proof} We have
	\begin{equation*}
	\begin{aligned}
    \|f_{\Delta,L}^{\alpha}-L(f)\|_{\infty} &= \| f_{\Delta,L}^{\alpha}- f + f - L(f) \|_{\infty}\\
	&\leq \| f_{\Delta,L}^{\alpha} - f\|_{\infty} + \|f - L(f)\|_{\infty}\\
	&\leq \|\alpha\|_{\infty}\| f_{\Delta,L}^{\alpha}- L(f)\|_{\infty} + \|f - L(f)\|_{\infty}\\
	\end{aligned}
	\end{equation*}
	Thus,
	$$\| f_{\Delta,L}^{\alpha}-L(f)\|_{\infty} \leq \frac{1}{1-\|\alpha\|_{\infty}}\|f-L(f)\|_{\infty},$$
	completing the proof.
\end{proof}

\begin{corollary}\label{cor2}
	Let $f \in \text{Lip}(I \times J)$ be the germ function and  $f_{\Delta,L}^{\alpha}$  be the $\alpha$-fractal function  associated with $f$ corresponding to $\alpha$, $\Delta$ and $L$. Then we have the following inequality.
	$$\| f_{\Delta,L}^{\alpha}-f\|_{\infty} \leq \frac{\|\alpha\|_{\infty}}{1-\|\alpha\|_{\infty}}\|f-L(f)\|_{\infty}.$$
\end{corollary}
\begin{proof}
	From Proposition  \ref{prop5.7}  and the triangle inequality
	\begin{equation*}
	\begin{aligned}
	\| f_{\Delta,L}^{\alpha}-f\|_{\infty} & \le \|\alpha\|_{\infty} \| f_{\Delta,L}^{\alpha}  - L(f) \|_{\infty}\\
	&\leq \|\alpha\|_{\infty} \big(\| f_{\Delta,L}^{\alpha} - f\|_{\infty} + \|f - L(f)\|_{\infty}\big)
	\end{aligned}
	\end{equation*}	
	proving the claim. 
\end{proof}
As an immediate consequence of the previous corollary, we obtain sequences  of  fractal functions, as specified in the upcoming result, converging uniformly to a prescribed  bivariate Lipschitz continuous function. 
\begin{corollary}\label{cor3}
	Let $f \in \text{Lip}(I \times J)$ be the germ function. 
	\begin{enumerate}
		\item  Assume that the partition $\Delta$ and scale function $\alpha$ are fixed.  Let $(L_n)_{n \in \mathbb{N}}$ be a sequence of operators  on Lip($I\times J$) such that $(L_n(f)) (x_k,y_l)=f(x_k,y_l)$ for all $k,l\in\{0,\infty\}$ and for each $n \in \mathbb{N}.$ Further assume that $L_n(f)\rightarrow f$ uniformly.   Then the sequence $\big(f_{\Delta,L_n}^{\alpha}  \big)_{n \in \mathbb{N}}$ of $\alpha$-fractal functions associated to $f$  converges uniformly to $f.$ 
	
		\item  Assume that the partition $\Delta$ and the operator $L: \text{Lip}(I)\rightarrow \text{Lip}(I)$ are fixed.  Let $(\alpha^{n})_{n\in\mathbb{N}}$ be a sequence of scale functions, where $\alpha^n = (\alpha^n_{ij})_{i,j \in \mathbb{N} \times \mathbb{N}}$, be  such that $\|\alpha^n\|_{\infty}\rightarrow 0$ as $ n \to \infty.$ Then the sequence  $\big(f_{\Delta,L}^{\alpha^n}  \big)_{n \in \mathbb{N}}$ converges uniformly to $f.$		
	\end{enumerate}
\end{corollary}

\begin{proposition}\label{linear}
	If $L:\text{Lip}(I \times J)\rightarrow \text{Lip}(I \times J)$ is a linear operator, then $\mathcal{F}^{\alpha}_{\Delta,L}$ is a linear operator.
\end{proposition}
\begin{proof}
	The proof follows on lines similar to \cite{MN1}, and hence omitted.
\end{proof}
\begin{remark}
	As noted earlier,  $\alpha$-fractal function and associated fractal operator are well-studied notions, in the case  of a univariate germ function and with a finite number of sample points from it.  Further, in the literature, the fractal operator is studied with the standing assumption that the parameter map $L$ is a bounded linear operator (see, for example, \cite{MN1, MN2, SV,VC,VN}). Consequently, the fractal operator is investigated only within the confines of the theory of  bounded linear operators. Here, we do not assume $L$ to be linear or bounded, thereby enhancing the scope of the fractal operator. 
\end{remark}

\par Now let us recall some preliminaries from the theory of nonlinear operators and their perturbation theory of operators. These terminologies can be found in \cite{RW} and many of these are obvious modifications to the notions present in the well-known treatise for perturbation theory of linear operators by Kato \cite{Kato}.
\begin{definition}
	An operator $A:D(A)\subset X\rightarrow Y$ is said to be closed if its graph $G(A)=\{(x,A(x)): x\in D(A)\}$ is a closed subspace of $X\times Y.$ This is equivalent to saying that for any sequence $(x_n)_{n\in\mathbb{N}}\text{ in }D(A),~x_n\rightarrow x\text{ and } A(x_n)\rightarrow y\in Y\text{ implies } x\in D(A)\text{ and } A(x) = y.$
\end{definition}

\begin{definition}
	An operator $A:D(A)\subset X\rightarrow Y$ is said to be closable if it has a closed extension. That is,  there exists  a set $X_0,~ D(A)\subset X_0\subset X$  and a closed operator $\tilde{A}:X_0\subset X\rightarrow Y \text{ such that } \tilde{A}(x) = A(x)\text{ for all }x\in D(A).$
\end{definition}

\begin{definition}
	Let $A:D(A)\subset X\rightarrow Y~\text{and}~B:D(B)\subset X\rightarrow Y$ be two operators such that $D(B)\subset D(A).$ If for every sequence $(x_n)_{n\in\mathbb{N}} \text{ in }D(B)$ with $x_n\rightarrow x,~B(x_n)\rightarrow y\text{ and } A(x_n)\rightarrow z\text{ imply } x\in D(A)\text{ and } A(x) = z,$ then $A$ is said to be $B$-closed.
\end{definition}

\begin{definition}
	Let $A:D(A)\subset X\rightarrow Y~\text{and}~B:D(B)\subset X\rightarrow Y$ be two operators such that $D(B)\subset D(A).$ For every pair of sequences $(x_n)_{n\in\mathbb{N}},~(x'_n)_{n\in\mathbb{N}} \text{ in }D(B)$  having same limit in $X$, say, 
$x_n \to x$, $x_n' \to x$ such that $(B(x_n))_{n\in\mathbb{N}}\text{ and }(B(x'_n))_{n\in\mathbb{N}}$ also have the same limits, say, $y \in Y$, if
$$ A(x_n)\rightarrow z,~ A(x'_n)\rightarrow z' \Rightarrow  x\in D(A)~~\text{ and } ~~z' = z,$$ then $A$ is said to be $B$-closable.
\end{definition}
\begin{definition}
Let $X$ and $Y$ be normed linear spaces over the field $\mathbb{K}$ with norms $\|.\|_X$ and $\|.\|_Y$, respectively. Let $A:D(A)\subset X \to Y$ be an operator and $\mathcal{O}(X,Y)$  be the set of all operators from $X$ into $Y$. Define a non-negative extended real-valued function $p$ on $\mathcal{O}(X,Y)$ by 
$$ p(A) = \max \Big\{ \sup_{ x \in D(A), x \neq 0} \frac{ \|A(x)\|_Y}{\|x\|_X} , \|A(0)\|_Y\Big\} .$$
The non-negative extended real-valued function $p$ may be viewed as a generalization of the norm for the linear operators. 
\end{definition}
\begin{definition}
If $p(A) < \infty,$ then the operator $A$ is said to be norm-bounded operator and the quantity $p(A)$ is called the norm of $A.$
\end{definition}
The set of all norm-bounded operators from $X$ to $Y$ is denoted by $\mathcal{B}(X,Y).$ It is easy to verify that $\Big(\mathcal{B}(X,Y), p(\cdot)\Big)$ is a normed linear space.
\begin{definition}
	An operator $A:D(A)\subset X\rightarrow Y$ is said to be  topologically bounded if it maps bounded sets to bounded sets.
\end{definition}
\begin{remark}
In contrast to the case of linear operators, here the notions norm-bounded and topologically bounded are not equivalent. 
\end{remark}

\begin{definition}
	Let $A:D(A)\subset X\rightarrow Y~\text{and}~B:D(B)\subset X\rightarrow Y$ be two operators such that $D(B)\subset D(A).$ Then we say that $A$ is relatively (norm) bounded with respect to $B$ or simply $B$-bounded if f or some non-negative constants $a$ and $b$, the following inequality holds
	$$\|A(x)\|_Y\leq a\|x\|_X + b\|B(x)\|_Y, \quad \forall ~~x\in D(B).$$
	 The infimum of all values $b$ for which the above inequality is satisfied is called the $B$-bound of $A.$
\end{definition} 

\begin{definition}
	An operator $A:D(A)\subset X\rightarrow Y$ is said to be  Lipschitz if there exists a constant $M>0$ such that
	$$\|A(x) - A(y)\|_Y \leq M\|x-y\|_X, \quad \forall~~ x,y\in D(A).$$
	For a Lipschitz operator $A:D(A)\subset X\rightarrow Y,$ the Lipschitz constant of $A,$ denoted by $|A|,$ is defined by
	$$ |A| = \sup_{ x\neq y}\frac{\|A(x)-A(y)\|_Y}{\|x-y\|_X}.$$
\end{definition}

\begin{definition}
	Let $A:D(A)\subset X\rightarrow Y~\text{and}~B:D(B)\subset X\rightarrow Y$ be two operators such that $D(B)\subset D(A).$Then we say that $A$ is relatively Lipschitz with respect to $B$ or simply $B$-Lipschitz if the following inequality holds 
	$$\|A(x) - A(y)\|_Y \leq M_1\|x-y\|_X + M_2\|B(x)-B(y)\|_Y, \quad~~ \forall x,y\in D(A)$$
	for some non-negative constants $M_1$ and $M_2.$ The infimum of all such values of $M_2$  is called the $B$-Lipschitz constant of $A.$
\end{definition}
\begin{proposition}\label{propnew1}
	If $L:Lip(I \times J)\rightarrow Lip(I \times J)$ is a continuous nonlinear (not necessarly linear) operator,  then so is the operator $\mathcal{F}^{\alpha}_{\Delta,L}.$
\end{proposition}
\begin{proof}
	Let $(f_n)_{n\in\mathbb{N}}$ be a sequence of Lipschitz functions such that $f_n\rightarrow f\in Lip(I).$ Then using the functional equation and by routine computations, we have
\begin{equation*}	
\begin{aligned}
 \|\mathcal{F}^{\alpha}_{\Delta,L} (f_n)- \mathcal{F}^{\alpha}_{\Delta,L} (f) \|_{\infty}  \leq \frac{1}{1-\|\alpha\|_{\infty}}\|f_n - f\|_{\infty} + \frac{\|\alpha\|_{\infty}}{1-\|\alpha\|_{\infty}}\|L(f_n) - L(f)\|_{\infty}.
\end{aligned}
\end{equation*}
	The above inequality in conjection with the convergence of $(f_n)_{n\in\mathbb{N}}$ and the continuity of $L$ establishes that $ \mathcal{F}^{\alpha}_{\Delta,L} (f_n) \to  \mathcal{F}^{\alpha}_{\Delta,L} (f)$, from which the continuity of $\mathcal{F}^{\alpha}_{\Delta,L}$ follows. 
\end{proof}

\begin{proposition}
	The nonlinear  $\alpha$-fractal operator $\mathcal{F}^{\alpha}_{\Delta,L}:Lip(I \times J)\rightarrow \mathcal{C}(I \times J)$ is an $L$-bounded operator with the $L$-bound not exceeding $\frac{\|\alpha\|_{\infty}}{1-\|\alpha\|_{\infty}}.$
\end{proposition}
\begin{proof}
	In view of Corollary \ref{cor2} we have
	\begin{equation*}
	\|\mathcal{F}^{\alpha}_{\Delta,L}(f)-f\|_{\infty} \leq \frac{\|\alpha\|_{\infty}}{1-\|\alpha\|_{\infty}}[\|f\|_{\infty}+\|L(f)\|_{\infty}],
	\end{equation*}
	Thus
	\begin{equation*}
	\| \mathcal{F}^{\alpha}_{\Delta,L}(f) \|_{\infty} \leq \frac{1}{1-\|\alpha\|_{\infty}}\|f\|_{\infty}+\frac{\|\alpha\|_{\infty}}{1-\|\alpha\|_{\infty}}\|L(f)\|_{\infty},
	\end{equation*}
	proving the claim.
\end{proof}	
\begin{corollary}
If $L$ is a norm-bounded nonlinear operator, that is, $p(L) < \infty$, then $\mathcal{F}^{\alpha}_{\Delta,L}$ is also norm-bounded. 
\end{corollary} 
\begin{proof}
From the previous proposition, we have  $$\| \mathcal{F}^{\alpha}_{\Delta,L}(f) \|_{\infty} \leq \frac{1}{1-\|\alpha\|_{\infty}}\|f\|_{\infty}+\frac{\|\alpha\|_{\infty}}{1-\|\alpha\|_{\infty}}\|L(f)\|_{\infty}.$$
Since $L$ is norm-bounded we have
$$ \sup_{f \in \text{Lip}(I \times J), f \neq 0} \frac{\|L(f)\|}{\|f\|} < p(L) < \infty.$$
Consequently, 
$$ \sup_{f \in \text{Lip}(I \times J), f \neq 0} \frac{\|\mathcal{F}^{\alpha}_{\Delta,L}(f) \|}{\|f\|} \le \frac{1}{1-\|\alpha\|_{\infty}}+\frac{\|\alpha\|_{\infty}}{1-\|\alpha\|_{\infty}} p(L).$$
Hence we have

$$ p (\mathcal{F}^{\alpha}_{\Delta,L}):= \max\Big \{ \sup_{f \in \text{Lip}(I \times J), f \neq 0} \frac{\|\mathcal{F}^{\alpha}_{\Delta,L}(f) \|}{\|f\|},  \|\mathcal{F}^{\alpha}_{\Delta,L}(0)\| \Big\} < \infty,$$
and the claim. 
\end{proof} 
Similarly, one can prove the following.  
\begin{corollary}
If $L$ is a topologically bounded operator, then so is the fractal operator $\mathcal{F}^{\alpha}_{\Delta,L}$. 

\end{corollary}

\begin{proposition}
	The operator $\mathcal{F}^{\alpha}_{\Delta,L}$ is $L$-Lipschitz with the $L$-Lipschitz constant less than or equal to  $\frac{\|\alpha\|_{\infty}}{1-\|\alpha\|_{\infty}}.$
\end{proposition}
\begin{proof}
	Using computations similar to that in Proposition \ref{propnew1}
	 \begin{equation}\label{lipeq}
	 \|\mathcal{F}^{\alpha}_{\Delta,L}(f) -  \mathcal{F}^{\alpha}_{\Delta,L}(g) \|_{\infty}\leq \frac{1}{1-\|\alpha\|_{\infty}}\|f - g\|_{\infty}+\frac{\|\alpha\|_{\infty}}{1-\|\alpha\|_{\infty}}\|L(f) - L(g)\|_{\infty},
	 \end{equation}
	 proving the claim.
\end{proof}
\begin{corollary}\label{Lipschitz}
If $L:\text{Lip}(I \times J)\subset \mathcal{C}(I \times J)\rightarrow \text{Lip}(I \times J)$ is a Lipschitz operator, then so is the fractal operator $\mathcal{F}^{\alpha}_{\Delta,L}$. Further,   $|\mathcal{F}^{(\alpha)}_{\Delta,L}| = \frac{1+\|\alpha\|_{\infty}|L|}{1-\|\alpha\|_{\infty}}.$
\end{corollary}
\begin{proof}
Follows directly from the above proposition.
\end{proof}

\begin{proposition}
	If $L:\text{Lip}(I \times J)\subset \mathcal{C}(I \times J)\rightarrow \text{Lip}(I \times J)$ is a Cauchy-continuous operator (that is, $L$ maps a Cauchy sequence to a Cauchy sequence), then $\mathcal{F}^{\alpha}_{\Delta,L}$ can not be a closed operator.
\end{proposition}
\begin{proof}
	Let us assume on the contrary that the fractal operator $\mathcal{F}^{\alpha}_{\Delta,L}$  is closed. Pick a function $f \in \mathcal{C}(I \times J)\backslash \text{Lip}(I \times J)$ and consider a sequence $(f_n)_{n \in \mathbb{N}}$ in Lip$(I \times J)$ such the $f_n \to f .$  Since $L$ is a 
Cauchy-continuous operator and $(f_n)_{n \in \mathbb{N}}$ is a Cauchy sequence in Lip$(I \times J)$ it follows that $(L(f_n))_{n \in \mathbb{N}}$ is a Cauchy sequence. Using 
$$ \|\mathcal{F}^{\alpha}_{\Delta,L}(f_n) -  \mathcal{F}^{\alpha}_{\Delta,L}(f_m) \|_{\infty}\leq \frac{1}{1-\|\alpha\|_{\infty}}\|f_n - f_m\|_{\infty}+\frac{\|\alpha\|_{\infty}}{1-\|\alpha\|_{\infty}}\|L(f_n) - L(f_m)\|_{\infty},$$
one can infer that $\big(\mathcal{F}^{\alpha}_{\Delta,L}(f_n) \big)_{n \in \mathbb{N}}$ is a Cauchy sequence in the Banach space $\mathcal{C}(I \times J)$. Assume that $ \mathcal{F}^{\alpha}_{\Delta,L}(f_n) \to  g$.  By the fact that  $\mathcal{F}^{\alpha}_{\Delta,L}$  is closed it follows that $f \in \text{Lip} (I \times J)$ and $g= \mathcal{F} ^{\alpha}_{\Delta,L} (f)$, which contradicts the choice of $f$. 
\end{proof}

\begin{proposition}
If $L:\text{Lip}(I \times J)\rightarrow \text{Lip}(I \times J)$ is a closed operator, then the operator $\mathcal{F}^{\alpha}_{\Delta,L}$ is $L$-closed.
\end{proposition}

\begin{proof}
Let $(f_n)_{n\in\mathbb{N}}$ be a sequence of Lipschitz functions such that $f_n \to f$, $ L(f_n) \to g$ and  $\mathcal{F}^{\alpha}_{\Delta,L}(f_n) \to h$ as $n \to \infty.$ Since $L$ is a closed operator, $f_n \to f$, $ L(f_n) \to g$ together imply $f \in \text{Lip} (I \times J)$ and 
$g= L(f)$.  By (\ref{lipeq}), we have
$$\| \mathcal{F}^{\alpha}_{\Delta,L}  (f_n) - \mathcal{F}^{\alpha}_{\Delta,L}(f)\|_{\infty} \leq \frac{1}{1-\|\alpha\|_{\infty}}\|f_n - f\|_{\infty} + \frac{\|\alpha\|_{\infty}}{1-\|\alpha\|_{\infty}}\|L(f_n) - L(f)\|_{\infty},$$
and consequently $\mathcal{F}^{\alpha}_{\Delta,L}  (f_n)  \to  \mathcal{F}^{\alpha}_{\Delta,L}(f)$ as $n \to \infty$. By the uniqueness of the limit, $h=  \mathcal{F}^{\alpha}_{\Delta,L}(f)$ and hence the assertion. 
\end{proof}	

Simlarly one can prove
\begin{proposition}
If $L:\text{Lip}(I \times J)\rightarrow \text{Lip}(I \times J)$ is a closable operator, then the operator $\mathcal{F}^{\alpha}_{\Delta,L}$ is $L$-closable.
\end{proposition}
\section{Extension of Fractal Operator and Some Properties}
In this short section we extend the fractal operator $\mathcal{F}^{\alpha}_{\Delta,L}$ to the whole of $\mathcal{C}(I \times J),$ and we shall refer to this extension operator as the $\alpha$-fractal operator on $\mathcal{C}(I \times J).$ The following lemma is a standard result (see also \cite{Vis}).
\begin{lemma}
	Let $A:D(A)\subset X\rightarrow Y$ be a Lipschitz operator with Lipschitz constant $|A|,$ where $X$ and $Y$ are metric spaces with $Y$ being complete. Then there exists a Lipschitz extension $\tilde{A}:\overline{D(A)}\subset X\rightarrow Y$ of $A:D(A)\subset X\rightarrow Y$ such that $|\tilde{A}| = |A|.$
\end{lemma} 
The operators $L:\text{Lip}(I \times J)\rightarrow \mathcal{C}(I \times J)$  and $\mathcal{F}^{\alpha}_{\Delta,L}:\text{Lip}(I \times J)\rightarrow \mathcal{C}(I \times J)$ are densely defined.  Therefore, if $L$ is a Lipschitz operator, then by the above lemma and Corollary \ref{Lipschitz},  we have Lipschitz extensions $\tilde{L}: \mathcal{C}(I \times J)\rightarrow \mathcal{C}(I \times J)$ of $L$ and $\tilde{\mathcal{F}}^{\alpha}_{\Delta,L}:\mathcal{C}(I \times J)\rightarrow \mathcal{C}(I \times J)$ of $\mathcal{F}^{\alpha}_{\Delta,L}$ preserving their respective Lipschitz constants. By a slight abuse of notation we denote this extension of the fractal operator $\mathcal{F}^{\alpha}_{\Delta,L}$ also by  $\mathcal{F}^{\alpha}_{\Delta,L}.$ This observation is formally recorded in the following proposition.
\begin{proposition}
	If $L:\text{Lip}(I \times J)\subset \mathcal{C}(I \times J)\rightarrow \text{Lip}(I \times J)$ is a Lipschitz operator, then the $\alpha$-fractal operator $\mathcal{F}^{\alpha}_{\Delta,L}: \text{Lip}(I \times J)\subset \mathcal{C}(I \times J)\rightarrow \mathcal{C}(I \times J)$ has a Lipschitz extension $\mathcal{F}^{\alpha}_{\Delta,L}:\mathcal{C}(I \times J)\rightarrow \mathcal{C}(I \times J)$ with a Lipschitz constant $|\mathcal{F}^{\alpha}_{\Delta,L}| = \dfrac{1+\|\alpha\|_{\infty}|L|}{1-\|\alpha\|_{\infty}}.$
\end{proposition}

\begin{definition}
	For a prescribed \textit{germ} function $f \in \mathcal{C}(I \times J,)$ the function $f^{\alpha}_{\Delta,L} = \mathcal{F}^{\alpha}_{\Delta,L}(f)$, where $\mathcal{F}^{\alpha}_{\Delta,L}$ is the Lipschitz extension in the previous proposition,  is called the (countable bivariate) $\alpha$-fractal function associated to the germ function $f$ with respect to the parameters $\alpha,~\Delta$ and $L.$
\end{definition}
\begin{remark}
	For each $f\in \mathcal{C}(I \times J)$, its fractal counterpart $f^{\alpha}_{\Delta,L}=  \mathcal{F}^{\alpha}_{\Delta,L}(f)$ satisfies a self-referential functional equation similar to (\ref{FPeq}).  To establish this,  let us take a sequence $(f_n)$ in $\text{Lip}(I \times J)$ converging uniformly to $f$. Then for $(x,y)\in I_i \times J_j = [x_{i-1},x_i] \times [y_{j-1},y_j]$, $,j=1,2,\ldots$ we have
	\begin{equation}\label{self1}
	\begin{aligned}
	f^{\alpha}_{\Delta,L}(x,y) =& \lim_{ n \to \infty} (f_n)^{\alpha}_{\Delta,L} (x,y) \\ 
 =& \lim_{n\to\infty}\Big[f_n(x,y)+\alpha_{ij}(x,y) \big\{(f_n)^{\alpha}_{\Delta,L}(u_i^{-1}(x),v_j^{-1}(y)) \\&- L(f_n)(u_i^{-1}(x),v_j^{-1}(y))\big\}\Big]\\
	&= f(x,y)+\alpha_{ij}(x,y)\big[f^{\alpha}_{\Delta,L}(u_i^{-1}(x),v_j^{-1}(y)) - \tilde{L}(f)(u_i^{-1}(x,v_j^{-1}(y)))\big].
	\end{aligned}
	\end{equation}
	Using the continuity of $f$ and $f^{\alpha}_{\Delta,L}$ we get a similar self-similar equation at other  points in $I \times J$.
\end{remark}
\begin{remark}
	If $L$ is a bounded linear operator, then $\mathcal{F}^{\alpha}_{\Delta,L}:\mathcal{C}(I \times J)\rightarrow \mathcal{C}(I \times J)$ is also a bounded linear operator. In the setting of univariate functions, the bounded linear fractal operator  $\mathcal{F}^{\alpha}_{\Delta,L}:\mathcal{C}(I)\rightarrow \mathcal{C}(I)$ is well-studied, see, for example, \cite{MN1,MN2,VN,VN2}.
\end{remark}
The notion of an invariant subspace is fundamental in operator theory. Next we provide a class of proper closed invariant subspaces for the bounded linear fractal operator $\mathcal{F}^{\alpha}_{\Delta,L}: \mathcal{C}(I \times J)\rightarrow \mathcal{C}(I \times J)$. Let us recall relevant standard results and definitions from basic theory of linear operators. 

\begin{definition}
	Let $X, Y$ be normed spaces and $T:X \rightarrow Y$ be a bounded linear operator. The adjoint or dual $T^*$ of $T$ is the unique map $T^*:Y^* \rightarrow X^*$ defined by
	$$T^*(\psi) = \psi\circ T \quad \text{ for all } \psi \in Y^*.$$
\end{definition}

\begin{definition}
	Given a Banach space $X,$ the annihilator (or pre-annihilator) of a subset $S$ of $X^*$ is the subspace  defined as follows
	$$S^{\perp} = \{x \in X: \quad \psi(x) = 0, ~~\forall~ \psi \in S\}.$$
\end{definition}
\begin{lemma}\label{invariant}
	Let $X$ be a non-separable Banach space and $T:X \rightarrow X$ be a bounded linear operator. Then,  for any non-zero element $x \in X$ the subspace $\overline{span\{x,T(x),T^{2}(x),\ldots\}}$ is a non-trivial closed invariant subspace. 
\end{lemma}

\begin{lemma}\label{perp}
	Let $X$ be a Banach space and $T:X \rightarrow X$ be a bounded linear operator. If $Y$ is a closed invariant subspace of the operator $T^*:X^* \rightarrow X^*,$ then $Y^{\perp}$ is a closed subspace which is invariant under $T.$ 
\end{lemma}
Recall  that the dual of $\mathcal{C}([a,b] \times [c,d])$  is isomorphic to the space of all Borel measures equipped with total variation norm and it is a non-separable Banach space. Using this fact and the above lemmas we construct a class of non-trivial closed invariant subspaces for the bounded linear fractal operator $\mathcal{F}^{\alpha}_{\Delta,L}.$ Consider the non-zero linear functional $\psi_{a,c}\in (\mathcal{C}([a,b]\times [c,d]))^*$ given by $\psi_{(a,c)}(f) = f(a,c)$ for all $f \in \mathcal{C}([a,b] \times [c,d]).$

\begin{theorem}
	The subspace $W_{(a,c)} := \{f \in \mathcal{C}([a,b] \times [c,d]): f(a,c) = 0\}$ is invariant for the fractal operator $\mathcal{F}^{\alpha}_{\Delta,L}$ for permissible choices of $\alpha$, $\Delta$ and bounded linear operator $L.$
\end{theorem}
\begin{proof}
	For any partition $\Delta$, scale function $\alpha$,  and bounded linear operator $L$ and any $f \in \mathcal{C}([a,b]\times [c,d])$ we have 
	\begin{equation*}
	\begin{aligned}
	(\mathcal{F}^{\alpha}_{\Delta,L})^*(\psi_{(a,c)})(f) &= \psi_{(a,c)}\circ \mathcal{F}^{\alpha}_{\Delta,L}(f)\\
	&= \psi_{(a,c)}(f^{\alpha}_{\Delta,L}) \\
	&= f^{\alpha}_{\Delta,L}(a,c)\\
	&= f(a,c)\\
	&= \psi_{(a,c)}(f)
	\end{aligned}
	\end{equation*}
	Consequently, $((\mathcal{F}^{\alpha}_{\Delta,L})^*)^n(\psi_{(a,c)}) = \psi_{(a,c)} ~~\forall~~n \in \mathbb{N}.$ and hence 
	by Lemma \ref{invariant}, $\overline{\text{span}\{ \psi_{(a,c)},(\mathcal{F}^{\alpha}_{\Delta,L})^*(\psi_{(a,c)}),    ((\mathcal{F}^{\alpha}_{\Delta,L})^*)^2(\psi_{(a,c)}),\ldots\}}=\text{span}\{ \psi_{(a,c)}\}$ is a non-trivial closed invariant subspace of $(\mathcal{C}([a,b]\times [c,d]))^*$. Now from Lemma \ref{perp} it follows that  $(\text{span}\{ \psi_{(a,c)}\})^{\perp} = W_{(a,c)}$ is a non-trivial closed invariant subspace for $\mathcal{F}^{\alpha}_{\Delta,L}$.
\end{proof}
\begin{remark}
Along with the proof of the previous theorem, it is perhaps worth recalling  that  the closed invariant subspace $Y^{\perp}$ in Lemma \ref{perp}  can be trivial. For example, consider the space $X=\ell_1$ and take $Y=c_0$, the space of
all sequences convergent to zero. Then $Y$ is a closed invariant subspace of the identity operator $Id:\ell_{\infty} \to \ell_{\infty}$. However,  the preannihilator of $Y$ is zero.

\end{remark}

\begin{theorem}
	Let $\Delta = \{x_i: i\in \mathbb{N}_0\} \times \{y_j: j\in \mathbb{N}_0\}$ be a partition of the rectangle $[a,b] \times [c,d]$ and $L$ be a bounded linear operator. Then $W := \{f \in \mathcal{C}([a,b] \times [c,d]): f(x_i,y_j) = 0,~ \forall~(i,j) \in \mathbb{N}_0\times\mathbb{N}_0, f(a,c) = f(a,d) = f(b,c) = f(b,d) = 0\}$ is a non-trivial closed invariant subspace for the $\alpha$-fractal operator $\mathcal{F}^{\alpha}_{\Delta,L}.$
\end{theorem}

\begin{proof}
	On lines similar to the above theorem we have, $W_{(i,j)} := \big\{f \in \mathcal{C}([a,b] \times [c,d]): f(x_i,y_j) = 0\big\}$  for all $(i,j) \in \mathbb{N}_0\times\mathbb{N}_0$ , $W_{(a,d)}$, $W_{(b,c)}$  and $W_{(b,d)}$ are non-trivial closed invariant subspaces of $\mathcal{F}^{\alpha}_{\Delta,L}.$ The result is immediate by taking the intersection of these subspaces.
\end{proof}

\bibliographystyle{amsplain}

\begin{thebibliography}{10}
	
	\bibitem {MF1} M. F. Barnsley, \textit{Fractal functions and interpolation}, Constr. Approx.,  2(1986) 303-32.
	\bibitem {B}  M. F. Barnsley, \textit{Fractals Everywhere}, Academic Press, Orlando, Florida, 1988.

\bibitem{BD}  P. Bouboulis, L. Dalla,  \textit{A general construction of fractal interpolation functions on grids of $\mathbb{R}^n$}, Eur. J. Appl. Math., 18 (4) (2007) 449-476.
\bibitem {GP} A. K. B. Chand, G. P. Kapoor, \textit{Hidden variable bivariate fractal interpolation surfaces}, Fractals, 11 (2003) 277-288.
\bibitem{SC}  S. Chen, \textit{The non-differentiability of a class of fractal interpolation functions}, Acta Math. Sci.,  19(4) (1999) 425-430. 
\bibitem {Dalla} L. Dalla, \textit{Bivariate fractal interpolation functions on grids}, Fractals, 10 (2002) 53-58.

\bibitem{DO} A. Deniz, Y. \"{ O}zdemir, \textit{Graph-directed fractal interpolation functions}, Turk J. Math., 41 (2017) 829-840. 

	\bibitem {Feng}  Z. Feng, \textit{Variation and Minkowski dimension of fractal interpolation surfaces}, J. Math. Anal. Appl., 176(1993) 561-586.
	\bibitem {GH} J. S. Geronimo,  D. Hardin, \textit{Fractal interpolation surfaces and a related 2D multiresolution analysis }, J. Math. Anal. Appl., 176 (1993) 561-586.

\bibitem{HM} D. Hardin, P. R. Massopust, \textit{ Fractal interpolation functions from $\mathbb{R}^n$ into $\mathbb{R}^m$ and their projections}, Z. Anal. Anwend, 12 (1993) 535-548. 
	\bibitem{JH} J. Hutchinson, \textit{Fractals and self-similarity}, Indiana Univ. Math. J.,  30(1981) 713-747.

\bibitem{SCNA} S. Jha, A. K. B. Chand, M. A. Navascu\'es, A. Sahu, \textit{Approximation properties of bivariate $\alpha$-fractal functions and dimension results}, Applicable Analysis,  (2020) DOI: 10.1080/00036811.2020.1721472.
	\bibitem{Kato} T. Kato, \textit{Perturbation Theory for Linear Operators}, Springer-Verlag, New York, 1980.
\bibitem{Luor} D-C. Luor, \textit{Fractal interpolation functions with partial self similarity}, J. Math. Anal. Appl., 464(1) (2018) 911-923. 
	\bibitem {Mal} R. Malysz, \textit{The Minkowski dimension of the bivariate fractal interpolation surfaces}, Chaos Solitons \& Fractals, 27 (2006)  27-50.
	\bibitem {PM} P. R. Massopust, \textit{Fractal surfaces}, J. Math. Anal. Appl., 151 (1990)  275-290.
\bibitem{PM1} P. R. Massopus, \textit{Non-stationary fractal interpolation}, Mathematics, 7(8) (2019) 666.
\bibitem{PM2} P. R. Massopust, \textit{ Fractal Functions, Fractal Surfaces, and Wavelets}, Academic Press, 2nd ed., 2016.
	\bibitem {Metz} W. Metzer, C. H. Yun, \textit{Construction of fractal interpolation surfaces on rectangular grids}, Internat. J. Bifur. Chaos, 20 (2010) 4079-4086.
\bibitem {MN1} M. A. Navascu\'es, Fractal polynomial interpolation, Z. Anal. Anwend., 25(2) (2005) 401-418.
\bibitem {MN2} M. A. Navascu\'es, \textit{Fractal approximation}, Complex Anal. Oper. Theory, 4(4) (2010)  953-974.
\bibitem{MN3}  M. A. Navascu\'es, \textit{Fractal trigonometric approximation}, Electronic Transactions  Num. Anal.,  20 (2005) 64-74. 
\bibitem{SR} S. Ri, \textit{A new nonlinear fractal interpolation function}, Fractals, 25(6) (2017) 1750063.
	\bibitem{RW} H. M.  Riedl, G. F. Webb, \textit{Relative boundedness conditions and the perturbations of nonlinear operators,}  Czechoslovak Math. J.,  24 (1974) 584-597.
	\bibitem {Ruan} H-J.  Ruan,  Q. Xu, \textit{Fractal interpolation surfaces on rectangular grids}, Bull. Aust. Math. Soc., 91 (2015)  435-446.
	\bibitem{NAS1} N. A. Secelean, \textit{The existence of the attractor of countable iterated function systems},  Mediterr. J. Math.,  9 (1) (2012)  61–79.
	\bibitem{NAS2} N. A. Secelean, \textit{The fractal interpolation for countable systems of data}, Univ. Beograd. Publ. Elektrotehn. Fak. Ser. Mat.,  14 (2003) 11–19. 
	\bibitem{SV} S. Verma,  P. Viswanathan, \textit{A fractal operator associated with bivariate fractal interpolation functions
		on rectangular grids}, Results Math., 75(28) (2020).
	
\bibitem{VC} P. Viswanathan, A. K. B. Chand, \textit{Fractal rational functions and their approximation properties}, J. Approx. Theory,  185 (2014) 31-50. 
\bibitem {VN}  P. Viswanathan,  M. A. Navascu\'es, \textit{A Fractal operator on some standard spaces of functions}, Proc. Edin. Math. Soc., 60 (2017) 771-786.

\bibitem{VN2}   P. Viswanathan, A. K. B. Chand, M. A. Navascu\'es, \textit{Fractal perturbation preserving fundamental shapes:  Bounds on the scale factors}, J. Math. Anal. Appl., 419 (2014) 804-817. 
\bibitem{Vis} P. Viswanathan, \textit{Fractal approximation of a function from a countable sample set and associated fractal operator}, RACSAM 114, 32 (2020). https://doi.org/10.1007/s13398-019-00772-8.
\bibitem{WY} H-Y. Wang, J-S. Yu, \textit{Fractal interpolation functions with variable parameters and their analytical properties}, J. Approx. Theory, 175 (2013) 1-18. 
	\bibitem {Xie} H. Xie,  H. Sun, \textit{The study on bivariate fractal interpolation functions and creation of fractal interpolated surfaces}, Fractals, 5 (1997) 625-634.
	\bibitem {Zhao} N. Zhao, \textit{Construction and application of fractal interpolation surfaces}, Visual Computer, 12 (1996) 132-146.
	
\end{thebibliography}

\end{document}